\documentclass[12pt, reqno]{amsart}
\usepackage{amsmath, amsthm, amscd, amsfonts, amssymb, xcolor}
\usepackage{bm}

\newtheorem{theorem}{Theorem}[section]
\newtheorem{lemma}[theorem]{Lemma}

\theoremstyle{definition}
\newtheorem{definition}[theorem]{Definition}

\numberwithin{equation}{section}

\newcommand{\vp}{\varphi}

\newcommand{\clb}{\mathcal{B}}

\newcommand{\cle}{\mathcal{E}}

\newcommand{\clh}{\mathcal{H}}
\newcommand{\clk}{\mathcal{K}}

\newcommand{\cls}{\mathcal{S}}

\newcommand{\D}{\mathbb{D}}

\newcommand{\T}{\mathbb{T}}
\newcommand{\Z}{\mathbb{Z}}

\begin{document}


\setcounter{page}{1}


\title[Paired and Toeplitz + Hankel operators]{Paired and Toeplitz + Hankel operators}

\author[Das]{Nilanjan Das}
\address{Indian Statistical Institute, Statistics and Mathematics Unit, 8th Mile, Mysore Road, Bangalore, 560059,
India}
\email{nilanjand7@gmail.com}

\author[Das]{Soma Das}
\address{Indian Statistical Institute, Statistics and Mathematics Unit, 8th Mile, Mysore Road, Bangalore, 560059, India}
\email{dsoma994@gmail.com}

\author[Sarkar]{Jaydeb Sarkar}
\address{Indian Statistical Institute, Statistics and Mathematics Unit, 8th Mile, Mysore Road, Bangalore, 560059,
India}
\email{jay@isibang.ac.in, jaydeb@gmail.com}



\begin{abstract}
We present complete classifications of Toeplitz + Hankel operators on vector-valued Hardy spaces and classify paired operators on $L^2(\mathbb{T})$. We also study the latter class through the lens of inner functions on the disc.
\end{abstract}

\maketitle

\tableofcontents

\section{Introduction}\label{sec intro}

This paper focuses on two classical classes of operators: Toeplitz operators and Hankel operators.  Although Toeplitz and Hankel operators are distinguishable topics, this paper will examine them together by focusing on bounded linear operators $T$ that admit the form
\[
T = \text{Toeplitz operator} + \text{Hankel operator}.
\]
In short, we will write these operators as Toeplitz + Hankel. Our objective here is to provide thorough classifications of Toeplitz + Hankel on vector-valued Hardy spaces.

Given a Hilbert space $\cle$ (in this paper, all Hilbert spaces are separable and over the field $\mathbb{C}$ of complex numbers), $L^2_{\cle}(\T)$ is defined as the Lebesgue-Bochner space of all $\cle$-valued measurable functions $f: \T \to \cle$ such that
\[
\|f\|:= \Big(\frac{1}{2 \pi} \int_{0}^{2 \pi} \|f(e^{i\theta})\|^2 d\theta \Big)^\frac{1}{2} < \infty,
\]
where $\T:=\{z\in\mathbb{C}:|z|=1\}$.
The $\cle$-valued Hardy space $H^2_{\cle}(\T)$ is the closed subspace of $L^2_{\cle}(\T)$ that is made up of functions of $L^2_{\cle}(\T)$ whose negative Fourier coefficients are equal to zero \cite[Chapter IX]{FF}. Given the radial limits, $H^2_{\cle}(\T)$ is equivalent to the Hilbert space of analytic functions $f : \D \to \cle$, such that
\[
\|f\| := \sup_{0 < r < 1} \Big(\frac{1}{2 \pi} \int_{0}^{2 \pi} \|f(r e^{i\theta})\|^2 d\theta \Big)^\frac{1}{2} < \infty,
\]
where $\mathbb{D}:=\{z\in\mathbb{C}:|z|<1\}$.
Moreover, we define the Banach space $L^\infty_{\clb(\cle)}(\T)$ as the space of all measurable functions $\Phi: \T \to \clb(\cle)$ such that
\[
\|\Phi\|:= \underset{\theta \in [0, 2 \pi]}{\text{ess} \sup} \, \|\Phi(e^{i \theta})\| < \infty,
\]
where $\clb(\cle)$ is the space of all bounded linear operators on $\cle$. For each $\Phi \in L^\infty_{\clb(\cle)}(\T)$, the multiplication operator $M_\Phi: L^2_{\cle}(\T) \to L^2_{\cle}(\T)$ (the vector-valued Laurent operator) defined by
\[
M_\Phi f = \Phi f \qquad (f \in L^2_{\cle}(\T)),
\]
is a bounded linear operator, and $\|M_\Phi\| = \|\Phi\|$. The \textit{Toeplitz operator} $T_\Phi \in \clb(H^2_{\cle}(\T))$ with symbol $\Phi \in L^\infty_{\clb(\cle)}(\T)$ is defined by
\[
T_\Phi = P_{H^2_{\cle}(\T)} M_\Phi|_{H^2_{\cle}(\T)}.
\]
Let us clarify that by $P_S$, we denote the orthogonal projection of a Hilbert space $\clh$ onto a closed subspace $S$ of $\clh$.
We remark that the most basic Toeplitz operator that holds significant importance is $T_z$, which is associated with the symbol $\Phi = z I_\cle \in L^\infty_{\clb(\cle)}(\T)$. Therefore, we have $T_z = P_{H^2_{\cle}(\T)} M_z|_{H^2_{\cle}(\T)}$. Similarly, the \textit{Hankel operator} $H_\Phi \in \clb(H^2_{\cle}(\T))$ with symbol $\Phi \in L^\infty_{\clb(\cle)}(\T)$ is defined by
\[
H_\Phi = P_{H^2_{\cle}(\T)} M_\Phi J|_{H^2_{\cle}(\T)},
\]
where
\[
(Jf)(z) = f(\bar z) \qquad (f \in L^2_{\cle}(\T), z \in \T),
\]
is the standard conjugation operator on $L^2_{\cle}(\T)$. In Theorem \ref{thm TH vv}, we present classifications of Toeplitz + Hankel operators on $H^2_\cle(\T)$ as follows:

\begin{theorem}
Let $A \in \clb(H^2_{\cle}(\T))$. Then the following are equivalent:
\begin{enumerate}
\item $A = $ Toeplitz + Hankel.
\item $A T_z -T_z^*A$ is a Toeplitz operator.
\item $T_z^*A T_z -A$ is a Hankel operator.
\item $A T_z +{T_z^*}^2A T_z = T_z^*A +T_z^*A T_z^2$.
\end{enumerate}
\end{theorem}

When $\cle=\mathbb{C}$, Theorem \ref{thm TH vv} recovers the classification of Toeplitz + Hankel operators obtained in \cite{Ehr} (see the end of Section \ref{sec TH scalar} for more details).
We also contribute to the structure of scalar-valued Toeplitz + Hankel operators in terms of Beurling inner functions. Before proceeding further, let us mention that throughout this article, we will use the notations $L^2(\T), H^2(\T)$ and $L^\infty(\T)$ for the spaces $L^2_{\mathbb{C}}(\T), H^2_{\mathbb{C}}(\T)$ and $L^\infty_{\clb(\mathbb{C})}(\T)$, respectively. It is known that if the kernel of a Hankel operator is nonzero, then it is invariant under the shift $T_z$ and hence admits a Beurling representation. The converse of this statement is also true (see Section \ref{sec TH scalar} for more details). Therefore, non-injective Hankel operators are both substantial and significant when viewed from a function-theoretic standpoint. Consequently, it is anticipated that this class of Hankel operators will possess a more intricate structure. In view of this, in Theorem \ref{main-th-T+H}, we further offer a completely different perspective of the above theorem in the scalar case and non-injective Hankel operators, which exposes further function-theoretic aspects of the symbols:

\begin{theorem}
Let $A \in \clb(H^2(\T))$. Then $A$ is a sum of a Toeplitz and a non-injective Hankel operator if and only if $T_z^*A T_z -A$ is a non-injective Hankel operator. Additionally, if
\[
{T_z^*}A T_z -A = H_{\psi_1},
\]
for some $\psi_1 \in L^\infty(\T)$, then there exist an inner function $\theta \in H^\infty(\D)$ and a symbol $\vp_2 \in L^\infty(\T)$ such that
\[
\ker H_{\psi_1} = \theta H^2(\T),
\]
and
\[
A T_\theta = T_{\vp_2},
\]
and $A = T_\vp + H_\psi$ for some symbol $\psi \in L^\infty(\T)$, where $\vp = \overline{\theta} \vp_2$.
\end{theorem}

Here and hereafter, we denote by $H^\infty(\mathbb{D})$ the Banach algebra of scalar-valued bounded analytic functions defined in $\mathbb{D}$. Moreover, a function $\theta \in H^\infty(\D)$ is said to be \textit{inner} if the radial limit function of $\theta$ satisfies
\[
|\theta(z)| = 1,
\]
for almost all $z \in \T$.

We now proceed to the other contribution of this paper, which concerns paired operators. The idea of paired operators can be traced back at least to Widom's 1960s work in the context of singular integral equations \cite{HW} (also see Shinbrot \cite{MS}, and \cite{CG, Mikh, Pro}). However, the notion of paired operators can be presented in a much wider framework. Let $\clh$ be a Hilbert space and $P$ and $Q$ be orthogonal projections. Suppose
\[
P+Q = I_{\clh},
\]
that is, $P$ and $Q$ are complementary projections on $\clh$. Let $A, B \in \clb(\clh)$. The \textit{paired operator} corresponding to $A$ and $B$ and the complementary projections $P$ and $Q$ is defined by
\[
S_{A,B} = A P + B Q.
\]
The order of the projections in the above identity is significant, and in Section \ref{sec paired on L} (see also (\ref{eqn duals})), we shall explore the reverse order of projections in the context of transposed paired operators.

Throughout, $P_+$ will denote the orthogonal projection (also known as Szeg\"{o} projection) of $L^2(\T)$ onto the closed subspace $H^2(\T)$. Also, we write
\[
P_- = I_{L^2(\T)} - P_+.
\]
Therefore, $P_+$ and $P_-$ are complementary projections on $L^2(\T)$. The notion of paired operators that interests us can now be explained (cf. \cite{Cam}). 

\begin{definition}\label{def paired}
Given $\vp, \psi \in L^\infty(\T)$, the paired operator $S_{\vp, \psi}: L^2(\T)\to L^2(\T)$ is defined by
\[
S_{\vp, \psi} = M_\vp P_{+} + M_\psi P_{-}.
\]
\end{definition}

Theorem \ref{SNJ_Th1} provides an algebraic classification of paired operators, serving as another key result of this paper: Let $X\in\clb(L^2(\T))$. Then $X$ is a paired operator if and only if
\[
X = M_z^* X M_z P_{+} + M_z X M_z^*P_{-}.
\]
Moreover, if there exist $\vp, \psi \in L^\infty(\T)$ such that
\[
X=S_{\vp, \psi},
\]
then the representing functions $\vp$ and $\psi$ of $X$ are unique. In fact, the uniqueness of $\vp$ and $\psi$ follows easily from \cite[Corollary 2.4]{Cam}.

The idea of paired operators applies to a broader scope. In this paper, we explore the concept of $\theta$-paired operators ($\theta$ is an inner function) from a model space perspective. In particular, Theorem \ref{SNJ_Th2} of this paper presents a complete classification of $\theta$-paired operators. The idea of such paired operators is related to Sarason's notion of truncated Toeplitz operators (see Section \ref{sec remarks}).

In closing this section, we remark that, from a perturbation theory standpoint, the classification or general study of the Toeplitz + Hankel operators is also noteworthy. This is because one can view them as Hankel perturbations of Toeplitz operators, and vice versa. The same applies to paired operators. It is also worth noting that there is a considerable number of finite-rank Hankel operators.

The remainder of the paper is composed as follows: The paper's real contribution starts with the Toeplitz + Hankel operators. Section \ref{sec TandH} classifies this type of operator in the setting of vector-valued Hardy spaces. Section \ref{sec TH scalar} focuses on the scalar case and addresses the classification problem for sums of Toeplitz and non-injective Hankel operators. Section \ref{sec paired on L} deals with the classification of paired operators on $L^2(\T)$, while Section \ref{sec paired on H} introduces and then classifies $\theta$-paired operators on $H^2(\T)$. The final section, Section \ref{sec remarks}, consists of some general discussion concerning the content of this paper.

\section{Toeplitz + Hankel operators: the general case}\label{sec TandH}

This section focuses on the characterizations of Toeplitz + Hankel operators defined on vector-valued Hardy spaces. In Section \ref{sec TH scalar}, we will revisit this result in the context of the sum of scalar-valued Toeplitz and non-injective Hankel operators with a new proof and will provide further information about the symbols of such a class of operators.

Let $T, H \in \clb(H^2_{\cle}(\T))$. We have the following classifications of Toeplitz and Hankel operators \cite{FF, RR}: $T$ is a Toeplitz operator if and only if
\[
T_z^* T T_z = T,
\]
and $H$ is a Hankel operator if and only if
\[
T_z^* H = H T_z.
\]

We need a vector-valued variant of Lemma 12 from \cite{Ehr}. We merely rewrite the proof of \cite[Lemma 12]{Ehr} in the present setting.

\begin{lemma}\label{lma-Ehr}
Let $\cle$ be a Hilbert space, $M > 0$, and let $\Phi : \T \to \clb(\cle)$ be a measurable function. Suppose
\[
\|\Phi(z)(1-z^{2n})\|_{\infty} \leq M,
\]
for all $n \geq 1$ and $z \in \T$ a.e. Then $\Phi \in L^\infty_{\clb(\cle)}(\T)$.
\end{lemma}
\begin{proof}
Consider the set $X = \{e^{2 \pi i \alpha}: \alpha \in \mathbb{Q}^c\}$. Observe that the measure of $\T \setminus X$ is zero. For each $z \in X$, a well-known result of Kronecker and Jacobi \cite[page 38]{Nad} ensures that $\{z^{2n}: n \in\Z\}$ is dense in $\T$. For $z \in X$,  a.e. it follows that
\[
\begin{split}
\|\Phi(z)\| & = \Big\|\frac{1}{1 - z^{2n}} \Phi(z)(1-z^{2n})\Big\|
\\
& \leq \inf_{n \in \Z} \frac{M}{|1-z^{2n}|}
\\
& =  \frac{M}{\sup_{n \in \Z}|1-z^{2n}|}
\\
& =\frac{M}{2},
\end{split}
\]
as $\sup_{n \in \Z}|1-z^{2n}| = 2$. Therefore, $\Phi \in L^\infty_{\clb(\cle)}(\T)$.
\end{proof}

Now we are in a position to state and prove the Brown-Halmos type algebraic characterization of Toeplitz + Hankel operators defined on vector-valued Hardy spaces.

\begin{theorem}\label{thm TH vv}
Let $\cle$ be a Hilbert space, and let $A \in \clb(H^2_{\cle}(\T))$. The following conditions are equivalent:
\begin{enumerate}
\item $A$ is a Toeplitz + Hankel operator.
\item $A T_z - T_z^*A$ is a Toeplitz operator.
\item $T_z^*A T_z -A$ is a Hankel operator.
\item $A T_z + T_z^{*2} A T_z = T_z^* A + T_z^* A T_z^2$.
\end{enumerate}
\end{theorem}
\begin{proof}
First, we prove that (1) implies (4). Suppose there are $\Phi, \Psi \in L^\infty_{\clb(\cle)}(\T)$ such that
\[
A = T_\Phi + H_\Psi .
\]
We know that $T_z^* T_\Phi T_z = T_\Phi $ and $H_\Psi T_z =T_z^* H_\Psi$. We repeatedly exploit these identities and compute:
\[
\begin{split}
A T_z + T_z^{*2} A T_z & = T_\Phi T_z + H_\Psi T_z + T_z^{*2} (T_\Phi + H_\Psi) T_z
\\
& = T_\Phi T_z + T_z^* H_\Psi + T_z^*(T_z^* T_\Phi T_z) + T_z^* (T_z^* H_\Psi) T_z
\\
& = T_\Phi T_z + T_z^* H_\Psi + T_z^* T_\Phi + T_z^* (T_z^* H_\Psi) T_z
\\
& = T_z^* T_\Phi T_z^2 + T_z^* (H_\Psi + T_\Phi) + T_z^* H_\Psi T_z^2
\\
& = T_z^* (T_\Phi + H_\Psi) T_z^2 + T_z^* (H_\Psi + T_\Phi).
\end{split}
\]
We conclude that $A T_z + T_z^{*2} A T_z = T_z^* A + T_z^* A T_z^2$. For the converse, we assume that A satisfies the identity in (4). Then
\[
T_z^* (T_z^* A - A T_z) T_z = T_z^*A - A T_z,
\]
which implies that $T_z^* A - A T_z$ is a Toeplitz operator. Therefore, there exists $\Phi_1 \in L^\infty_{\clb(\cle)}(\T)$ such that
\[
T_z^*A-AT_z = T_{\Phi_1}.
\]
Multiplying this by $T_z$ and $T_z^*$, respectively, yields
\[
T_z^*AT_z -	AT_z^2  = T_{z\Phi_1},
\]
and
\[
{T_z^*}^2A-T_z^*AT_z =T_{\bar{z}\Phi_1}.
\]
By adding these two identities, we obtain
\[
{T_z^*}^2A - AT_z^2 = T_{z\Phi_1} + T_{\bar{z}\Phi_1} = T_{(z+\bar{z})\Phi_1},
\]
and hence
\[
{T_z^*}^2A - AT_z^2 = T_{\frac{z^2-\bar{z}^2}{z-\bar{z}}\Phi_1}.
\]
With this calculation and view, we now proceed by induction. Fix a natural number $N > 1$. Assume that
\[
{T_z^*}^nA -AT_z^n = T_{\frac{z^n-\bar{z}^n}{z-\bar{z}}\Phi_1},
\]
for all $1\leq n\leq N$. Then
\[
{T_z^*}^NAT_z -	AT_z^{N+1}  = T_{\frac{z^N-\bar{z}^N}{z-\bar{z}}z\Phi_1},
\]
and
\[
{T_z^*}^{N+1}A-T_z^*AT_z^N =T_{\frac{z^N-\bar{z}^N}{z-\bar{z}}\bar{z}\Phi_1}.
\]
Adding the above two identities yields
\[
{T_z^*}^{N+1}A - AT_z^{N+1} + T_z^*({T_z^*}^{N-1}A  - AT_z^{N-1})T_z = T_{\frac{z^N-\bar{z}^N}{z-\bar{z}}(z+\bar{z})\Phi_1}.
\]
By our induction hypothesis
\[
{T_z^*}^{N+1}A -AT_z^{N+1} + T_{\frac{z^{N-1}-\bar{z}^{N-1}}{z-\bar{z}}\Phi_1} = T_{\frac{z^N-\bar{z}^N}{z-\bar{z}}(z+\bar{z})\Phi_1},
\]
and hence
\[
\begin{split}
{T_z^*}^{N+1}A -AT_z^{N+1} & = T_{[(z^N-\bar{z}^N)(z+\bar{z})-(z^{N-1}-\bar{z}^{N-1})]\frac{\Phi_1}{z-\bar{z}}}
\\
& = T_{\frac{z^{N+1}-\bar{z}^{N+1}}{z-\bar{z}}\Phi_1}.
\end{split}
\]
Therefore, by induction, we conclude that
\[
{T_z^*}^nA -AT_z^n = T_{\frac{(z^n-\bar{z}^n)\Phi_1}{z-\bar{z}}},
\]
for all $n \in \Z_+$. By the norm identity of vector-valued Toeplitz operators \cite[Section 6.2, page 110]{RR}, we have
\[
\left\|T_{\frac{(z^n-\bar{z}^n)\Phi_1}{z-\bar{z}}}\right\| = \left\|\frac{z^n-\bar{z}^n}{z-\bar{z}}\Phi_1 \right\|_{\infty},
\]
and consequently
\[
\left\|\frac{z^n-\bar{z}^n}{z-\bar{z}}\Phi_1 \right\|_{\infty} \leq 2\|A\|.
\]
Moreover, observe that
\[
\begin{split}
\left\|\frac{z^n-\bar{z}^n}{z-\bar{z}}\Phi_1 \right\|_{\infty} & = \text{ess} \sup_{z\in \T}\left\|\frac{z^n-\bar{z}^n}{z-\bar{z}}\Phi_1(z)\right\|
\\
& = \text{ess} \sup_{z\in \T}\left\|\frac{1-z^{2n}}{z-\bar{z}}\Phi_1(z) \right\|.
\end{split}
\]
Lemma \ref{lma-Ehr} yields
\[
\Phi:= \frac{1}{\bar{z}-z} \Phi_1 \in L^\infty_{\clb(\cle)}(\T).
\]
It is clear from the definition of $\Phi$  that $T_{\Phi_1} = T_{(\bar{z}-z)\Phi}$, which implies
\[
T_{\Phi_1} = T_z^*T_{\Phi} - T_{\Phi} T_z.
\]
On the other hand, by the definition of $\Phi_1$, we know that $T_{\Phi_1} = T_z^* A - A T_z$, and hence
\[
T_z^*T_{\Phi} - T_{\Phi} T_z = T_z^*A-AT_z,
\]
or equivalently
\[
T_z^*(A- T_{\Phi})= (A- T_{\Phi})T_z.
\]
This is the same as saying that $A- T_{\Phi}$ is a Hankel operator. Therefore, there exists $\Psi \in L^\infty_{\clb(\cle)}(\T)$ such that $A = T_{\Phi} + H_{\Psi}$. This proves (1) and completes the proof of the equivalence of (1) and (4).

\noindent The equivalence between (2) and (4) follows immediately once we rewrite condition (4) as
\[
T_z^* (A T_z - T_z^*A) T_z = A T_z - T_z^*A.
\]
Indeed, the above identity is equivalent to the statement that $A T_z - T_z^*A$ is a Toeplitz operator. Similarly, condition (4), rewritten as
\[
(T_z^* A T_z - A) T_z = T_z^* (T_z^* A T_z - A),
\]
leads directly to the equivalence between (3) and (4).
\end{proof}

Now we turn to the uniqueness of the symbols representing Toeplitz + Hankel operators: Let $\cle$ be a Hilbert space and let $A\in \clb (H^2_{\cle}(\T))$. If $A$ is a Toeplitz + Hankel operator, then there exist a unique symbol $\Phi \in L^\infty_{\clb(\cle)}(\T)$ and a symbol $\Psi \in L^\infty_{\clb(\cle)}(\T)$ unique up to translation by functions from $\overline{zH^\infty_{\clb(\cle)}(\D)}$ such that
\[
A = T_\Phi + H_\Psi,
\]
where $\overline{zH^\infty_{\clb(\cle)}(\D)}$ denotes the set of all $\Phi \in L^\infty_{\clb(\cle)}(\T)$ of the form
\[
\Phi = \sum_{n=1}^{\infty} \Phi_n e^{- i n \theta},
\]
where $\Phi_n \in \clb(\cle)$ for all $n \geq 1$.

\noindent \textit{Proof of the claim:} Suppose there exist symbols $\Phi, \Phi_1, \Psi, \Psi_1 \in L^\infty_{\clb(\cle)}(\T)$ such that
\[
A = T_\Phi + H_\Psi = T_{\Phi_1} + H_{\Psi_1}.
\]
Since
\[
A T_z - T_z^* A = T_\Phi T_z - T_z^* T_\Phi = T_{\Phi_1} T_z - T_z^* T_{\Phi_1},
\]
and, in general
\[
T_{\chi} T_z - T_z^* T_{\chi} = T_{(z - \bar z) \chi} \qquad (\chi \in L^\infty_{\clb(\cle)}(\T)),
\]
it follows that $T_{(z-\bar{z})(\Phi -\Phi_1) }=0$. Then
\[
\|T_{(z-\bar{z})(\Phi -\Phi_1) }\| = \|(z-\bar{z})(\Phi -\Phi_1)\|_\infty,
\]
implies that
\[
(z-\bar{z})(\Phi -\Phi_1) =0,
\]
and hence $\Phi = \Phi_1 $ almost everywhere on $\T$. Now we have that $H_\Psi= H_{\Psi_1}$. Then \cite[Lemma 3.2, Chapter IX]{FF} yields
\[
\Psi - \Psi_1 \in \overline{zH^\infty_{\clb(\cle)}(\D)}.
\]
This completes the proof of the claim.

\section{Toeplitz + non-injective Hankel operators}\label{sec TH scalar}

Hankel operators that are not injective are particularly interesting. It is well known and also will be outlined soon that the kernel of a Hankel operator is shift invariant. Therefore, if the kernel of a Hankel operator is nonzero (that is, if the Hankel operator is not injective), then it is automatically associated with an inner function. In light of this perspective, this section offers yet another proof of the classifications of the scalar-valued Toeplitz + non-injective Hankel operators. Although establishing a weaker version of the general classification is something that should worry one right away, there are at least two good reasons for doing so. First, this special case's proof is based more on classical function theory. Second, the function theory-based element of the proof establishes enhanced interactions among the symbols of the Toeplitz and the non-injective Hankel operators.

The first component of the function-theoretic approach is the inner function-based representations of closed subspaces of the Hardy space. Recall that a function $\theta \in H^\infty(\D)$ is said to be inner if the radial limit function of $\theta$ satisfies the following condition (see \cite[Chapter 3]{FM} for further details):
\[
|\theta(z)| = 1,
\]
for all $z \in \T$ a.e. Inner functions are of importance in the theory of the Hardy space because, among many other reasons, they parameterize invariant subspaces of the shift operators. To be more specific, a nonzero closed subspace $\cls \subseteq H^2(\T)$ is invariant under $T_z$ if and only if there is an inner function $\theta \in H^\infty(\D)$ that makes
\[
\cls = \theta H^2(\T).
\]
This inner function is unique up to the multiplication of the circle group $\T$. This is one of Beurling's most famous and well-known results \cite{Beurling}. We note that $\theta$ inner makes the (analytic) Toeplitz operator $T_\theta$ isometric, implying right away that the subspace $\theta H^2(\T)$ is closed.

The second tool is a distinctive feature of Hankel operators. Recall that a bounded linear operator $H$ on $H^2(\T)$ is Hankel if and only if
\[
T_z^*H = H T_z.
\]
In particular, if $H_\vp$ is a Hankel operator with the symbol $\vp \in L^\infty(\T)$, then the above intertwiner property in particular implies that
\[
T_z \ker H_\vp \subseteq \ker H_\vp,
\]
that is, $\ker H_\vp$ is invariant under the shift $T_z$ on $H^2(\T)$. By the Beurling theorem, if $\ker H_\vp\neq\{0\}$ then there exists an inner function $\theta \in H^\infty(\D)$ such that
\begin{equation}\label{eqn  ker H}
\ker H_\vp = \theta H^2(\T).
\end{equation}
On the other hand, it is well known (cf. \cite[page 15]{Peller book}) that for any inner function $\theta \in H^\infty(\D)$, there is always a symbol $\vp \in L^\infty(\T)$ such that \eqref{eqn  ker H} holds.

In the scalar case, Theorem \ref{thm TH vv} now modifies as follows:

\begin{theorem}\label{main-th-T+H}
Let $A \in \clb(H^2(\T))$. Then $A$ is a sum of a Toeplitz and a non-injective Hankel operator if and only if $T_z^*A T_z -A$ is a non-injective Hankel operator. Additionally, if
\[
{T_z^*}A T_z -A = H_{\psi_1},
\]
for some $\psi_1 \in L^\infty(\T)$, then there exist an inner function $\theta \in H^\infty(\D)$ and a symbol $\vp_2 \in L^\infty(\T)$ such that
\[
\ker H_{\psi_1} = \theta H^2(\T),
\]
and
\[
A T_\theta = T_{\vp_2},
\]
and $A = T_\vp + H_\psi$ for some symbol $\psi \in L^\infty(\T)$, where $\vp = \overline{\theta} \vp_2$.
\end{theorem}
\begin{proof}
Suppose $A= T_\vp + H_\psi$ for some symbols $\vp, \psi \in L^\infty(\T)$, and assume that
\[
\ker H_\psi \neq \{0\}.
\]
Then
\[
T_z^*A T_z -A = T_z^*H_\psi T_z -H_\psi = H_{(\bar{z}^2 -1)\psi},
\]
which is a Hankel operator. Also, for any nonzero $f$ in $\ker H_\psi$, we have, in particular
\[
T_{z^2 -1}^*H_\psi(f) =0,
\]
which implies
\[
H_{(\bar{z}^2 -1)\psi} f =0.
\]
Consequently
\[
\ker (T_z^*A T_z -A) \neq \{0\}.
\]
This means $T_z^*A T_z -A$ is a non-injective Hankel operator.

\noindent Conversely, suppose $T_z^*A T_z -A = H_{\psi_1}$ for some symbol $\psi_1 \in L^\infty(\T)$. Assume that $H_{\psi_1}$ is non-injective. By the discussion preceding the statement of this theorem, or more specifically by \eqref{eqn  ker H}, there exists a nonconstant inner function $\theta \in H^\infty(\D)$ such that
\[
\ker H_{\psi_1}  = \theta H^2(\T).
\]
Therefore, for any $f\in H^2(\T)$, we know that $H_{\psi_1}(\theta f) = 0$, which is equivalent to saying that
\[
H_{\psi_1}T_\theta=0.
\]
Consequently
\[
(T_z^*A T_z -A) T_\theta = H_{\psi_1}T_\theta = 0,
\]
and hence
\[
T_z^*A T_\theta T_z = A T_\theta.
\]
This says that $A T_\theta$ is a Toeplitz operator. Then there exists $\vp_2 \in L^\infty(\T)$ such that
\begin{equation}\label{T_2}
A T_\theta = T_{\vp_2}.
\end{equation}
Moreover,
\[
T_z^*(T_z^*A T_z -A) -(T_z^*A T_z -A) T_z= 0,
\]
implies that
\[
T_z^*(A T_z -T_z^*A) T_z = A T_z -T_z^*A.
\]
Therefore, there exists $\vp_1 \in L^\infty(\T)$ such that
\[
T_{\vp_1} = A T_z -T_z^*A.
\]
By multiplying the identity $T_{\vp_1} = A T_z -T_z^*A$ by $T_\theta$ from the right-hand side and then using the above equation \eqref{T_2}, we obtain
\begin{align*}
T_{\vp_1\theta} & = (A T_z -T_z^*A) T_\theta
\\
& = AT_\theta T_z-T_z^*AT_\theta
\\
& = T_{\vp_2} T_z -T_z^* T_{\vp_2}.
\end{align*}
Making use of the fact that $\theta$ is an inner function, we have the representation of $T_{\vp_1}$ as
\[
T_{\vp_1} =  T_{\bar{\theta}\vp_2} T_z -T_z^* T_{\bar{\theta}\vp_2}.
\]
Since $T_{\vp_1} = A T_z - T_z^* A$, it follows that
\[
A T_z - T_z^* A = T_{\bar{\theta}\vp_2} T_z -T_z^* T_{\bar{\theta}\vp_2},
\]
and hence
\[
(A - T_{\bar{\theta}\vp_2}) T_z = T_z^* (A - T_{\bar{\theta}\vp_2}).
\]
That is, $(A - T_{\bar{\theta}\vp_2})$ is a Hankel operator. Therefore, there exists a symbol $\psi \in L^\infty(\T)$ such that
\[
H_\psi = (A - T_{\bar{\theta}\vp_2}),
\]
which implies $A = H_\psi + T_{\bar{\theta}\vp_2}$. Thus $A = H_\psi + T_\vp$, where
\[
\vp = \bar{\theta}\vp_2 \in L^\infty(\T).
\]
It is left to show that $H_\psi$ is non-injective. Let us just observe that
\[
T_z^*A T_z -A = H_{(\bar{z}^2 -1)\psi}= H_{\psi_1}.
\]
That means
\[
\ker H_{(\bar{z}^2 -1)\psi} \neq \{0\}.
\]
For any nonzero $f \in \ker H_{(\bar{z}^2 -1)\psi}$, we have
\[
H_{\psi}T_{{z}^2 -1} f =0,
\]
from which, we conclude that
\[
\ker H_\psi \neq \{0\}.
\]
In the preceding, we exploited the understanding that $T_{{z}^2 -1}$ is an analytic Toeplitz operator, meaning 
\[
T_{{z}^2 -1} f \neq 0,
\]
as $f\neq 0$. Hence, the proof is complete.
\end{proof}

Clearly, the representing symbols $\vp$ and $\psi$ of the Toeplitz + Hankel operator in the above theorem are more informative. This is due to function theoretic techniques, which are less effective in the vector-valued case.

Now we will take a look into the scalar version (that is, $\cle = \mathbb{C}$) of Theorem \ref{thm TH vv}'s equivalence of (1) and (4) through the lens of the inner product. In view of the orthonormal basis $\{z^n\}_{n \in \Z_+}$ of $H^2(\T)$, it is easy to see that condition (4) is equivalent to the following identity:
\[
\langle A z^m, z^{n-1} \rangle + \langle A z^m, z^{n+1} \rangle = \langle A z^{m-1}, z^n \rangle + \langle A z^{m+1}, z^n\rangle,
\]
for all $m, n \geq 1$. This identity has been recently proved by Ehrhardt, Hagger, and Virtanen in \cite[Theorem 20]{Ehr}. For finite matrices, this was earlier obtained by Bevilacqua, Bonannie, and Bozzo \cite{Bevil} (also see Strang and MacNamara \cite[Section 8]{GS}): An $N\times N$ complex matrix $A=(a_{i,j})_{i, j=1}^N$ is Toeplitz + Hankel if and only if
$$
a_{i-1, j}+a_{i+1, j}=a_{i, j-1}+a_{i, j+1}
$$
for all $1< i, j<N$.

We note that our proofs of classifications of scalar (the non-injective Hankel operator case) and vector-valued Toeplitz + Hankel operators restricted to the above particular situations deviate significantly from the respective proofs. We add, however, that our proof of Theorem \ref{thm TH vv} uses a vector-valued version of a lemma from \cite{Ehr} (see Lemma \ref{lma-Ehr} for more details).

\section{Paired operators on $L^2(\T)$}\label{sec paired on L}

The goal of this section is to fully classify paired operators acting on $L^2(\T)$. Recall that $P_+$ denotes the orthogonal projection of $L^2(\T)$ onto the closed subspace $H^2(\T) \subseteq L^2(\T)$, and $P_- = I_{L^2(\T)} - P_+$. Therefore
\[
P_- = P_{H^2(\T)^\perp}.
\]
Now we observe a general fact about projections $P_+$ and $P_-$. Since $M_z(H^2(\T))\subseteq H^2(\T)$, it follows that
\begin{equation}\label{eqn PzP}
P_- M_z P_+ = 0.
\end{equation}
We will use this identity in what follows. We also need to recollect a standard fact about certain bounded linear operators on $L^2(\T)$ \cite[Lemma 8.12]{FM}: Let $\vp \in L^2(\T)$. If there exists $M > 0$ such that
\[
\|p \vp\| \leq M \|p\|,
\]
for all trigonometric polynomials $p$, then $\vp \in L^\infty(\T)$, and
\begin{equation}\label{eqn M vp}
\|M_\vp\| = \|\vp\|_\infty \leq M.
\end{equation}
Now we are ready for the classification of paired operators on $L^2(\T)$.

\begin{theorem}\label{SNJ_Th1}
Let $X \in \clb(L^2(\T))$. Then $X$ is a paired operator if and only if
\[
X= M_z^* X M_z P_{+} + M_z X M_z^*P_{-}.
\]
Moreover, if there exist $\vp, \psi \in L^\infty(\T)$ such that
\[
X=S_{\vp, \psi},
\]
then the symbols $\vp$ and $\psi$ in the representation of $X$ are unique.
\end{theorem}
\begin{proof}
Assume first that $X=S_{\vp, \psi}$ for some $\vp, \psi \in L^\infty(\T)$. In  view of \eqref{eqn PzP}, we have
\begin{align*}
M_z^* S_{\vp, \psi} M_z P_ + & = M_{\overline{z}}(M_\vp P_{+} + M_\psi P_{-}) M_z P_+
\\
& = M_{\overline{z}} M_\vp P_+ M_z P_+
\\
& = M_\vp P_+.
\end{align*}
Also, by \eqref{eqn PzP}, we have $P_+ M_{\overline{z}} P_- = 0$, and hence
\begin{align*}
M_z S_{\vp, \psi} M_z^* P_- & = M_{z} (M_\vp P_{+} + M_\psi P_{-}) M_{\overline{z}} P_-
\\
& = M_{z} M_\psi P_ - M_{\overline{z}} P_-
\\
& = M_\psi P_-.
\end{align*}
It is now evident that
\[
\begin{split}
S_{\vp, \psi} & = M_\vp P_+ + M_\psi P_-
\\
& = M_z^* S_{\vp, \psi} M_z P_+ + M_z S_{\vp, \psi} M_z^* P_-,
\end{split}
\]
which proves the necessary part of the theorem. For the reverse direction, assume that
\[
X= M_z^* X M_z P_{+} + M_z X M_z^*P_{-}.
\]
It follows that $X P_+ = M_z^* X M_z P_{+}$, and hence
\[
X M_z P_+ = M_z X P_+.
\]
By using this repeatedly, we obtain, for each $n\geq 1$, that
\[
\begin{split}
X(z^n) & = X P_+ M_z^n(1)
\\
& = X M_z P_+ M_z^{n-1}(1)
\\
& = M_z X P_+ M_z^{n-1}(1),
\end{split}
\]
and hence $X(z^n) = M_z^n X(1)$. Set
\[
\vp :=X(1).
\]
Now, since $X z^n = z^n \vp$ for all $n \geq 0$, we have
\[
X p = p \vp \qquad (p \in \mathbb{C}[z]).
\]
Let us now note that for a given trigonometric polynomial $p \in L^2(\T)$, there exists $m\in \mathbb{Z}_+$ such that
\[
z^m p \in \mathbb{C}[z] \subseteq H^2(\T).
\]
Moreover, observe, in general, that
\[
\|z^mg\|_2 = \|g\|_2,
\]
for all $g\in L^2(\T)$. Therefore, for a trigonometric polynomial $p \in L^2(\T)$, with the above setting, we compute
\[
\begin{split}
\|\vp p \|_2 & = \|\vp z^m p \|_2
\\
& =\|X(z^mp)\|_2
\\
& \leq\|X\|\|z^mp\|_2
\\
& =\|X\|\|p\|_2.
\end{split}
\]
By the fact recorded in \eqref{eqn M vp}, we immediately conclude that $\vp \in L^\infty(\T)$. Fix $f\in H^2(\T)$. There exists a sequence of polynomials $\{p_n\}_{n=1}^\infty \subseteq H^2(\T)$ such that
\[
p_n \longrightarrow f,
\]
in $L^2(\T)$. Therefore
\begin{align*}
\|X f - \vp f\|_2 & \leq \|X f - X p_n \|_2 + \|X p_n - \vp p_n\|_2 + \|\vp p_n - \vp f\|_2
\\
& \leq\|X\| \|f-p_n\|_2 + \|\vp p_n - \vp p_n\|_2 + \|\vp\|_\infty \|p_n-f\|_2
\\
&=(\|X\|+\|\vp\|_\infty)\|p_n-f\|_2
\\
& \longrightarrow 0.
\end{align*}
This implies
\[
X f = \vp f \qquad (f\in H^2(\T)),
\]
or, in other words
\begin{equation}\label{SNJ_Eq4}
XP_+=M_\vp P_+.
\end{equation}
Once again, we observe that $X= M_z^* X M_z P_{+} + M_z X M_z^*P_{-}$ implies that
\[
XP_- = M_z X M_z^* P_{-},
\]
from which
it immediately follows that
\[
X M_z^* P_- = M_z^* X P_-.
\]
For each $n \geq 1$, we use the same procedure as before to calculate
\[
\begin{split}
X(\overline{z}^n) & = X M_{z}^{* n-1} P_- (\overline{z})
\\
& =M_{z}^{* n-1} X P_- (\overline{z})
\\
& = M_{z}^{* n-1} X \overline{z}
\\
& = M_{z}^{* n} (z X \overline{z}).
\end{split}
\]
By setting
\[
\psi := z X(\overline{z}),
\]
we have $X(\overline{z}^n) = \psi \overline{z}^n$ for all $n \geq 1$. Clearly, for any polynomial $p\in(H^2(\T))^\perp$, we have
\[
X(p)=\psi p.
\]
Again, for a fixed polynomial $p\in L^2(\T)$, there exists $m\in \Z_+$ such that
\[
\overline{z}^mp\in(H^2(\T))^\perp.
\]
Since $\|\overline{z}^mg\|_2=\|g\|_2$ for all $g\in L^2(\T)$, it follows that
\[
\begin{split}
\|\psi p\|_2 & = \|\psi (\overline{z}^mp)\|_2
\\
& = \|X(\overline{z}^mp)\|_2
\\
& \leq \|X\|\|\overline{z}^mp\|_2
\\
& =\|X\|\|p\|_2.
\end{split}
\]
Therefore, \eqref{eqn M vp} guarantees that $\psi \in L^\infty(\T)$. Furthermore, it is easy to see that for any $f\in (H^2(\T))^\perp=\overline{zH^2(\T)}$, there exists a sequence of polynomials $\{p_n\}_{n=1}^\infty \subseteq (H^2(\T))^\perp$ such that
\[
p_n \longrightarrow f,
\]
in $L^2(\T)$. From here onwards, we can follow exactly similar lines of arguments as in the proof of $(\ref{SNJ_Eq4})$ to conclude that
\[
X f = \psi f \qquad (f\in(H^2(\T))^\perp),
\]
that is, $XP_-=M_\psi P_-$. Combining this and $(\ref{SNJ_Eq4})$, we have
\[
\begin{split}
X & = X (P_+ + P_-)
\\
& =M_\vp P_+ + M_\psi P_-,
\end{split}
\]
which proves the sufficient part of the theorem. For the uniqueness part, suppose $X \in \clb(L^2(\T))$ is a paired operator, and suppose
\[
X=S_{\vp, \psi} =S_{\vp_1, \psi_1},
\]
for some $\vp, \psi, \vp_1, \psi_1\in L^\infty(\T)$. Then
\[
S_{\vp-\vp_1,\psi-\psi_1}=0,
\]
and it follows from \cite[Corollary 2.4]{Cam} that $\vp = \vp_1$, $\psi=\psi_1$.
\end{proof}

Given $\vp, \psi \in L^\infty(\T)$, the \emph{transposed paired operator} $\Sigma_{\vp, \psi}:L^2(\T)\to L^2(\T)$ is defined by (cf. \cite{Cam})
\[
\Sigma_{\vp, \psi} = P_+ M_\vp + P_- M_\psi.
\]
$S_{\vp, \psi}$ and $\Sigma_{\vp, \psi}$ are called \textit{dual} to each other due to the relation (cf. \cite[Proposition 2.1]{Cam})
\[
S_{\vp, \psi}^*=\Sigma_{\bar{\vp}, \bar{\psi}}.
\]

As an immediate consequence of Theorem \ref{SNJ_Th1}, we have the following: Let $X\in\clb(L^2(\T))$. Then $X$ is a transposed paired operator if and only if
\[
X= P_{+}M_z^*XM_z+ P_{-}M_zXM_z^*.
\]
Moreover, if $X=\Sigma_{\vp, \psi}$ for $\vp, \psi \in L^\infty(\T)$, then $\vp$ and $\psi$ are unique representing functions.

\section{Paired operators on $H^2(\T)$}\label{sec paired on H}

The purpose of this section is to examine the concept of paired operators on $H^2(\T)$, or more specifically, paired operators corresponding to model spaces. Model spaces of $H^2(\T)$ are essentially $T_z^*$-invariant closed subspaces of $H^2(\T)$. We start our discussion by revisiting the Beurling theorem from Section \ref{sec TH scalar}: A nonzero closed subspace $\cls \subseteq H^2(\T)$ is $T_z$-invariant if and only if there is an inner function $\theta \in H^\infty(\D)$ (unique up to the multiplication of the group $\T$) such that
\[
\cls = \theta H^2(\T).
\]
The subsequent orthocomplement is then a closed subspace that is invariant under $T_z^*$. This space is also very significant for general research in Hilbert function spaces. We commonly refer to the space in question as the \textit{model space} associated with the inner function $\theta \in H^\infty(\D)$ and denote it by $\clk_\theta$. Therefore
\[
\clk_\theta = H^2(\T) \ominus \theta H^2(\T) \cong H^2(\T)/\theta H^2(\T).
\]
Consequently, given an inner function $\theta \in H^\infty(\D)$,  we have the splitting of the space $H^2(\T)$ as follows:
\[
H^2(\T) = \theta H^2(\T) \oplus \clk_\theta,
\]
and subsequently, we can define a new notion (and many more, just as similar to this) of paired operator. Throughout the rest of this section, $\theta\in H^\infty(\D)$ will be an arbitrary nonconstant inner function.

\begin{definition}
Let $\theta \in H^\infty(\D)$ be an inner function. The $\theta$-paired operator with (ordered) symbols $\vp, \psi \in H^\infty(\D)$ is defined by
\[
S_{\vp, \psi}^\theta = T_{\vp} P_{\theta H^2(\T)} + T_{\psi}P_{\clk_\theta}.
\]
\end{definition}

Clearly, $S_{\vp, \psi}^\theta: H^2(\T) \to H^2(\T)$ is a bounded linear operator. An operator $X \in \clb(H^2(\T))$ is said to be $\theta$-paired, if there exist $\vp$ and $\psi$ in $H^\infty(\D)$ such that
\[
X = S^\theta_{\vp, \psi}.
\]
Our goal is to characterize $\theta$-paired operators. Before we do that, let us go over a common and well-known fact about model spaces. Note that
\[
f \in H^2(\T) \cap \theta \overline{z H^2(\T)},
\]
if and only if $f \in H^2(\T)$ and
\[
f = \theta \bar{z} \bar{g},
\]
for some $g \in H^2(\T)$, which is further equivalent to
\[
M_\theta^* f = \bar{z} \bar{g} \in \overline{z H^2(\T)},
\]
for some $g \in H^2(\T)$. Therefore, we conclude that
\[
\clk_\theta = H^2(\T) \cap \theta \overline{z H^2(\T)}.
\]
Finally, one piece of notation. For a pair of operators $A, B \in \clb(\clh)$, the commutator $[A,B]$ is defined by
\[
[A,B] = AB - BA.
\]
Now we are ready for the characterization of $\theta$-paired operators.

\begin{theorem}\label{SNJ_Th2}
$X\in\clb(H^2(\T))$ is a $\theta$-paired operator if and only if the following two conditions are satisfied:
\begin{enumerate}
\item $\theta H^2(\T)$ is invariant under $X$.
\item There exists a nonzero $h_0\in \clk_\theta$ and $\nu \in H^\infty(\D)$ such that
\[
X h_0 = \nu h_0,
\]
and
\[
[X, T_z] = (X-T_{\nu}) P_{\theta H^2(\T)} T_z P_{\clk_\theta}.
\]
\end{enumerate}
Moreover, if $X=S_{\vp, \psi}^\theta $ is a $\theta$-paired operator, then the representing symbols $\vp, \psi \in H^\infty(\D)$ are unique for fixed $\theta$.
\end{theorem}
\begin{proof}
Suppose $X=S_{\vp, \psi}^\theta$ is a $\theta$-paired operator. For $h\in H^2(\T)$, we have
\[
S_{\vp, \psi}^\theta(\theta h) = \vp \theta h \in \theta H^2(\T),
\]
implying the invariance condition in (1). To verify (2), fix $f \in H^2(\T)$, and write
\[
f = \theta h_1 + \theta\bar{z} \overline{h_2},
\]
where $h_1, h_2\in H^2(\T)$ and $\theta\overline{zh_2}\in H^2(\T) \cap \theta\overline{zH^2(\T)}$. Now, we compute
\[
\begin{split}
S_{\vp, \psi}^\theta T_z f & = (T_\vp P_{\theta H^2(\T)} + T_\psi P_{\clk_\theta})(z \theta h_1 + z \theta \bar{z} \overline{h_2})
\\
& = z \vp \theta h_1 + T_\vp P_{\theta H^2(\T)} (z \theta \bar{z} \overline{h_2}) + T_\psi P_{\clk_\theta} (z \theta \bar{z} \overline{h_2})
\\
& = z \vp \theta h_1 + T_\vp P_{\theta H^2(\T)} (\theta \overline{h_2}) + T_\psi P_{\clk_\theta} (\theta \overline{h_2}),
\end{split}
\]
and, on the other hand, we also have
\[
\begin{split}
T_z S_{\vp, \psi}^\theta f & = T_z (T_\vp P_{\theta H^2(\T)} + T_\psi P_{\clk_\theta})(\theta h_1 + \theta \bar{z} \overline{h_2})
\\
& = T_z (T_\vp \theta h_1 + T_\psi \theta \bar{z} \overline{h_2})
\\
& = z \vp \theta h_1 + \psi \theta \overline{h_2}.
\end{split}
\]
As a result, we have
\begin{align*}
[X, T_z](f) & = S_{\vp, \psi}^\theta T_z f - T_z S_{\vp, \psi}^\theta f
\\
& = \vp P_{\theta H^2(\T)}(\theta \overline{h_2}) - \psi (I - P_{\clk_\theta}) (\theta \overline{h_2})
\\
& = \vp P_{\theta H^2(\T)}(\theta \overline{h_2}) - \psi P_{\theta H^2(\T)} (\theta \overline{h_2})
\\
& = X P_{\theta H^2(\T)}(\theta \overline{h_2}) - \psi P_{\theta H^2(\T)}(\theta \overline{h_2}).
\end{align*}
But
\[
P_{\theta H^2(\T)} T_z P_{\clk_\theta} f = P_{\theta H^2(\T)}(\theta \overline{h_2}),
\]
and hence
\[
[X, T_z](f) = (X-T_{\psi})P_{\theta H^2(\T)} T_z P_{\clk_\theta}(f).
\]
Define $\nu=\psi \in H^\infty(\D)$. Clearly, for all $h_0\in \clk_\theta$, we have $\nu h_0 = \psi h_0$. This proves the necessary part of the theorem. For the converse direction, assume that $X\in \clb(H^2(\T))$ satisfies the conditions (1) and (2). It is readily seen from condition (2) that $[X, T_z] P_{\theta H^2(\T)} = 0$, that is
\[
X T_z P_{\theta H^2(\T)} = T_z X P_{\theta H^2(\T)}.
\]
Since $P_{\theta H^2(\T)} = T_\theta T_{\overline{\theta}}$ (cf.  \cite[Theorem 14.11]{FM}), it follows that
\[
XT_zT_\theta T_{\overline{\theta}}=T_zXT_\theta T_{\overline{\theta}}.
\]
Multiplying both sides of the equality by $T_\theta$ from the right-hand side, we get
\[
XT_\theta T_z=T_zXT_\theta,
\]
and therefore there exists $\chi \in H^\infty(\mathbb{D})$ such that
\[
XT_\theta=T_\chi.
\]
At the same time, the fact that $X$ leaves $\theta H^2(\T)$ invariant ensures existence of $\vp \in H^2(\T)$ such that
\[
\chi = X \theta = \theta \vp.
\]
This asserts that $\vp\in L^\infty(\T)\cap H^2(\T)$, that is, $\vp\in H^\infty(\mathbb{D})$. As a result, we have
\[
\begin{split}
X T_\theta T_{\overline{\theta}} & = T_{\chi}T_{\overline{\theta}}
\\
& = T_{\vp\theta}T_{\overline{\theta}}
\\
& =T_\vp T_\theta T_{\overline{\theta}}.
\end{split}
\]
In other words
\begin{equation}\label{SNJ_Eq7}
XP_{\theta H^2(\T)}=T_{\vp}P_{\theta H^2(\T)}.
\end{equation}
Define $A\in\clb(H^2(\T))$ by
\[
A=XP_{\clk_\theta} + T_\nu P_{\theta H^2(\T)},
\]
where $\nu$ is as defined in the condition (2). Also, the second part of condition (2) implies
\[
(X T_z - T_z X)P_{\clk_\theta} = (X-T_{\nu}) P_{\theta H^2(\T)} T_z P_{\clk_\theta}.
\]
In particular, we have
\[
T_z X P_{\clk_\theta} = X T_z P_{\clk_\theta} - (X-T_{\nu}) P_{\theta H^2(\T)} T_z P_{\clk_\theta}.
\]
We compute
\begin{align*}
T_z A & = T_z (XP_{\clk_\theta}+T_\nu P_{\theta H^2(\T)})
\\
& = T_z X P_{\clk_\theta} + T_z T_\nu P_{\theta H^2(\T)}
\\
&=X T_z P_{\clk_\theta} - X P_{\theta H^2(\T)} T_z P_{\clk_\theta} + T_\nu P_{\theta H^2(\T)} T_z P_{\clk_\theta} + T_\nu T_z P_{\theta H^2(\T)}
\\
&= X (I - P_{\theta H^2(\T)}) T_z P_{\clk_\theta} + T_\nu P_{\theta H^2(\T)} T_z P_{\clk_\theta} + T_\nu T_z P_{\theta H^2(\T)}
\\
& = X P_{\clk_\theta} T_z P_{\clk_\theta} + T_\nu P_{\theta H^2(\T)} T_z P_{\clk_\theta} + T_\nu T_z P_{\theta H^2(\T)}
\\
& = X P_{\clk_\theta} T_z P_{\clk_\theta} + T_\nu P_{\theta H^2(\T)} T_z P_{\clk_\theta} + T_\nu P_{\theta H^2(\T)} T_z P_{\theta H^2(\T)}
\\
& = X P_{\clk_\theta} T_z P_{\clk_\theta} + T_\nu P_{\theta H^2(\T)} T_z P_{\clk_\theta} + (X P_{\clk_\theta} + T_\nu P_{\theta H^2(\T)}) T_z P_{\theta H^2(\T)}
\\
& = A T_z P_{\clk_\theta} + A T_z P_{\theta H^2(\T)}
\\
&=A T_z,
\end{align*}
and consequently, there exists $\psi \in H^\infty(\D)$ such that $A=T_{\psi}$. At the same time, we know that there exists a nonzero $h_0\in \clk_\theta$ such that $X(h_0)=\nu h_0$. Therefore
$$
\nu h_0 = X(h_0) = A(h_0) = T_{\psi}(h_0)= \psi h_0,
$$
which means $\psi = \nu$, and hence
\[
XP_{\clk_\theta}=T_\nu P_{\clk_\theta} = T_{\psi} P_{\clk_\theta}.
\]
Combining with $(\ref{SNJ_Eq7})$, this implies
$$
X=XP_{\theta H^2(\T)} + X P_{\clk_\theta} = T_{\vp}P_{\theta H^2(\T)}+T_{\psi} P_{\clk_\theta}.
$$
This settles the characterization part of the theorem. For the uniqueness part, assume that
\[
X = S_{\vp, \psi}^\theta = S_{\vp_1, \psi_1}^\theta,
\]
$\vp, \psi, \vp_1, \psi_1\in H^\infty(\D)$. Then
\[
X \theta = S_{\vp, \psi}^\theta \theta = S_{\vp_1, \psi_1}^\theta \theta,
\]
implies that
\[
\vp \theta = \vp_1 \theta,
\]
that is, $\vp = \vp_1$. Similarly, for a nonzero $h_0\in \clk_\theta$, we have
\[
X h_0 = S_{\vp, \psi}^\theta h_0 = S_{\vp_1, \psi_1}^\theta h_0,
\]
which yields $\psi = \psi_1$. Our proof is therefore complete.
\end{proof}

In the following section, we connect paired operators with truncated Toeplitz operators \cite{DS}.

\section{General remarks}\label{sec remarks}

In this concluding section, we draw some general remarks about Toeplitz operators, Hankel operators, and paired operators. Let us first identify one entry point for the Toeplitz + Hankel operators through paired operators. 

\subsection{Toeplitz + Hankel and paired operators} Note that Hankel operators can also be represented as compressions of Laurent operators, that is,
\[
H_\vp = P_- M_\vp|_{H^2(\T)},
\]
for $\vp \in L^\infty(\T)$. From this perspective, with respect to $L^2(\T) = H^2(\T) \oplus H^2(\T)^\perp$, we can say that
\[
M_\vp = \begin{bmatrix}
T_\vp & *
\\
H_\vp & *
\end{bmatrix}.
\]
Let us recall that given $\vp, \psi \in L^\infty(\T)$, the transposed paired operator $\Sigma_{\vp, \psi}$ is defined by (see the end of Section \ref{sec paired on L})
\begin{equation}\label{eqn duals}
\Sigma_{\vp, \psi} = P_+ M_\vp + P_- M_\psi.
\end{equation}
Now restricting the transposed paired operator $\Sigma_{\vp, \psi}$ to $H^2(\T)$, we get
\[
\Sigma_{\vp, \psi}|_{H^2(\T)} = P_+ M_\vp {\big|_{H^2(\T)}} + P_- M_\psi {\big|_{H^2(\T)}},
\]
that is
\begin{equation}\label{eqn po = TH}
\Sigma_{\vp, \psi}|_{H^2(\T)} = T_{\vp} + H_\psi \qquad (\vp, \psi \in L^\infty(\T)).
\end{equation}
Therefore,
\[
\Sigma_{\vp, \psi}|_{H^2(\T)} = \text{Toeplitz} + \text{Hankel}.
\]
These facts suggest a close relationship between paired operators and Toeplitz + Hankel operators. It is important to note that in this scenario, we treat the range spaces of $T_{\vp}$ and $H_\psi$ as subspaces of $L^2(\T)$.

\subsection{Truncated Toeplitz operators} We have already pointed out that the Toeplitz and the Hankel operators have a close association with the paired operators (see the identity in \eqref{eqn po = TH}). In particular, given a paired operator $S_{\vp, \psi}: L^2(\T)\to L^2(\T)$, one has the Toeplitz operator
\[
T_\vp = P_{H^2(\T)} S_{\vp, \psi}|_{H^2(\T)}.
\]
Similar to this, $\theta$-paired operators are associated with another important class of operators, namely, truncated Toeplitz operators. These are the compressions of Toeplitz operators on model spaces. More specifically, given an inner function $\theta \in H^\infty(\D)$, the \textit{truncated Toeplitz operator} with symbol $\vp \in L^\infty(\T)$ is defined by
\[
A_\vp^\theta f = P_{\clk_\theta} (\vp f) \qquad (f \in \clk_\theta).
\]
Therefore, for a $\theta$-paired operator $S_{\vp, \psi}^\theta: H^2(\T) \to H^2(\T)$ with $\vp, \psi \in H^\infty(\D)$, we have
\[
A_\psi^\theta = P_{\clk_\theta} S_{\vp, \psi}^\theta|_{\clk_\theta}.
\]
Since its inception by Sarason in \cite{DS}, numerous researchers have used this framework to formulate theories of diverse nature. At times, results in this direction are comparable to the classical theory of Toeplitz operators. For more details and further development, we refer the reader to \cite{Bar-2010, Bar-2011, CFT} and the references therein. 

In closing this subsection, we mention that our primary reason for studying $\theta$-paired operators is to see how they match up with some other ideas about paired operators (say, paired operators on $L^2(\T)$). The classifications of model space-based $\theta$-paired operators on $H^2(\T)$ derived in this work vividly suggest that the general theory of paired operators will shift from case to case. We finally remark that the potential impact of $\theta$-paired operators on the theory of truncated Toeplitz operators, or vice versa, is the subject of future investigation.

\subsection{Concluding remarks}
Before concluding this paper, we would like to provide some overarching comments. Although the Toeplitz + Hankel operators are among the most natural operators, research on their properties has only recently gained some momentum. Nevertheless, progress in understanding these operators has been somewhat slow, perhaps due to the potential intricacy even at the individual level of the Toeplitz and the Hankel operators (however, see \cite{Bevil, Deift, Didenko}). It is even challenging to anticipate the right results or the extent to which they will be pleasant and accessible. From this point of view, what makes this work distinctive is that it gives a comprehensive and clear classification of vector-valued Toeplitz + Hankel operators. The result appears to have more promise for future study in the theory of classical operators because it is provided in terms of algebraic property. We also address the scalar case of Toeplitz + non-injective Hankel operators, which incorporates additional function-theoretic techniques. This brings more distinctive flavor to the results in terms of Hilbert function space theory.

\vspace{0.3in}

\noindent\textbf{Acknowledgement:}
The first named author is supported by a post-doctoral fellowship provided by the National Board for Higher Mathematics (NBHM), India (sanction order no: 0204/34/2023/R\&D-II/16487, dated 15 December, 2023). The research of the second named author is supported by the Theoretical Statistics and Mathematics Unit, Indian Statistical Institute, Bangalore, India, and also by the post-doctoral fellowship provided by the National Board for Higher Mathematics (NBHM), India (Order No: 0204/16(8)/2024/R\&D- II/6760, dated May 09, 2024). The research of the third named author is supported in part by TARE (TAR/2022/000063) by SERB, Department of Science \& Technology (DST), Government of India.

\end{document}